\documentclass[11pt]{article}
\usepackage[margin=1in]{geometry}

\usepackage{authblk} 
\usepackage{graphicx}
\usepackage{amssymb}

\usepackage{amsmath}
\usepackage{enumerate}
\usepackage{amsfonts}
\usepackage[mathcal]{euscript}
\usepackage{amsthm}
\usepackage[all,2cell,import]{xy} \UseAllTwocells
\usepackage{xcolor}
\usepackage{tkz-euclide}
\usepackage{tikz}

\usepackage{url}
\makeatletter
\def\url@leostyle{%
	\@ifundefined{selectfont}{\def\UrlFont{\sf}}{\def\UrlFont{\small\ttfamily}}}
\makeatother
\usepackage[refpage]{nomencl}

\makenomenclature
\usepackage{makeidx}
\makeindex
\numberwithin{equation}{section}

\usepackage[colorlinks=true, pdfstartview=FitV, linkcolor=black,
			   citecolor=black, urlcolor=black]{hyperref}
\usepackage{bookmark}
\usepackage[titletoc]{appendix} 
\hypersetup{hypertexnames=false} 

          \newcommand{\nc}{\newcommand}
          \nc{\DMO}{\DeclareMathOperator}	
          \nc{\commentout}[1]{}

          \nc{\newnotation}{\nomenclature}
          \nc{\wrap}{\cW}
          \nc{\Tw}{\mathsf{Tw}}
          \nc{\loc}{\mathsf{Loc}}
          \nc{\Top}{Top}
          \nc{\emb}{\mathsf{emb}}
          \nc{\ind}{\mathsf{Ind}}
          \nc{\Ind}{\mathsf{Ind}}
          \nc{\Loc}{\mathsf{Loc}}
          \nc{\Cob}{\mathsf{Cob}}
          \nc{\mul}{\mathsf{Mul}}
          \nc{\fat}{\mathsf{fat}}
          \nc{\cob}{\mathsf{Cob}}
          \nc{\coh}{\mathsf{Coh}}
          \nc{\Liouaut}{\Aut}
          \nc{\Liouauto}{{\Aut^o}}
          \nc{\Liouautb}{\Aut^{b}}
          \nc{\Liouautgr}{\Aut^{gr}}
          \nc{\Liouautgrb}{\Aut^{gr,b}}
          \nc{\idem}{\mathsf{Idem}}
          \nc{\sets}{\mathsf{Sets}}
          \nc{\near}{\mathsf{near}}
          \nc{\sing}{\mathsf{Sing}}
          \nc{\Sing}{\mathsf{Sing}}
          \nc{\perf}{\mathsf{Perf}}
          \nc{\block}{\mathsf{block}}
          \nc{\ssets}{\mathsf{sSets}}
          \nc{\cmpct}{\mathsf{cmpct}}
          \nc{\compact}{\mathsf{cmpct}}
          \nc{\pwrap}{\mathsf{PWrap}}
          \nc{\coder}{\mathsf{Coder}}
          \nc{\bimod}{\mathsf{Bimod}}
          \nc{\grmod}{\mathsf{GrMod}}
          \nc{\Morita}{\mathsf{Morita}}
          \nc{\morita}{\mathsf{Morita}}
          \nc{\spaces}{\mathsf{Spaces}}
          \nc{\pwrms}{\mathsf{PWrFuk}_{M,S}}
          \nc{\pwrmf}{\mathsf{PWrFuk}_{M,F}}
          \nc{\pwrapmf}{\mathsf{PWrFuk}_{M,F}}
          \nc{\fuk}{\mathsf{Fukaya}}
          \nc{\infwr}{\mathsf{InfWr}}
          \nc{\fukaya}{\mathsf{Fukaya}}
          \nc{\autml}{\mathsf{Aut}_{M,\Lambda}}
          \nc{\fukml}{\mathsf{Fukaya}_{M,\Lambda}}
          \nc{\fukmle}{\mathsf{Fukaya}_{M,\Lambda,\epsilon}}
          \nc{\fukmod}{\wrfukcompact(M)\modules}
          \nc{\lag}{\mathsf{Lag}}
          \nc{\lagm}{\lag_M}
          \nc{\lago}{\lag^o}
          \nc{\lagml}{\lag_{M,\Lambda}} 
          \nc{\lagmle}{\lag_{M,\Lambda,\epsilon}}
          \nc{\Fun}{\mathsf{Fun}}
          \nc{\fun}{\mathsf{Fun}}
          \nc{\vect}{\mathsf{Vect}}
          \nc{\chain}{\mathsf{Chain}}
          \nc{\chainn}{Chain}
          \nc{\wrfuk}{\mathsf{WrFukaya}}
          \nc{\wrfukcompact}{\mathsf{WrFukaya}_{\mathsf{cmpct}}}
          \nc{\pwrfuk}{\mathsf{PWrFukaya}}
          \nc{\inffuk}{\mathsf{InfFuk}}
          \nc{\pwrfukml}{\mathsf{PWrFukaya}_{M,\Lambda}}
          \nc{\inffukml}{\mathsf{InfFuk}_{M,\Lambda}}
          \nc{\nattrans}{\mathsf{NatTrans}}
          \nc{\corres}{\mathsf{Corres}}
          \nc{\fukep}{\fukaya_\Lambda(M,\epsilon)}
          \nc{\fukepop}{\fukaya_\Lambda(M,\epsilon)^{\op}}
          \nc{\lagep}{\lag_\Lambda(M,\epsilon)}
          \DMO{\cyl}{cyl} 
          \nc{\dbcoh}{D^b\mathsf{Coh}}
          \nc{\corr}{\mathsf{Corr}}

          \nc{\cat}{\mathsf{Cat}}
          \nc{\Cat}{\mathsf{Cat}}
          \nc{\ainfty}{\mathsf{A}_\infty}
          \nc{\inftycat}{\mathcal{C}\!\operatorname{at}_\infty}
          \nc{\inftyCat}{\mathcal{C}\!\operatorname{at}_\infty}
          \nc{\inftyGpd}{\mathcal{G}\!\operatorname{pd}_\infty}
          \nc{\Ainftycat}{\mathcal{C}\!\operatorname{at}_{A_\infty}}
          \nc{\dgcat}{\mathcal{C}\!\operatorname{at}_{dg}}
          \nc{\ainftycat}{\mathcal{C}\!\operatorname{at}_{A_\infty}}
          \nc{\stablecat}{\mathcal{C}\!\operatorname{at}_\infty^{\Ex}}

          \DMO{\im}{im}
          \DMO{\ev}{ev}
          \DMO{\stable}{Ex}
          \DMO{\inj}{inj}
          \DMO{\fib}{fib}
          \DMO{\conf}{Conf}
          \DMO{\chains}{Chains}
          \DMO{\cochains}{Cochains}
          \DMO{\cone}{Cone}
          \DMO{\Map}{Map}
          \DMO{\ran}{Ran}
          \DMO{\rot}{Rot}
          \DMO{\leg}{Leg}
          \DMO{\imm}{imm}
          \DMO{\adj}{adj}
          \DMO{\symp}{Symp}
          \DMO{\tree}{Tree}
          \DMO{\cube}{Cube}
          \DMO{\weak}{weak}
          \DMO{\strong}{strong}
          \DMO{\Hoch}{Hoch}
          \DMO{\front}{front}
          \DMO{\flow}{Flow}
          \DMO{\floer}{Floer}
          \DMO{\Maps}{Maps}
          \DMO{\exact}{exact}
          \DMO{\excess}{Excess}
          \DMO{\Decomp}{Decomp}
          \DMO{\decomp}{Decomp}
          \DMO{\collar}{collar}
          \DMO{\yoneda}{Yoneda}
          \DMO{\hamspace}{Ham}
          \DMO{\sympspace}{Symp}
          \DMO{\holomaps}{Holomaps}
          \DMO{\comp}{Comp}
          \DMO{\crit}{Crit}
          \DMO{\test}{{test}}
          \DMO{\sign}{sign}
          \DMO{\topp}{top}
          \DMO{\indx}{Index}
          \DMO{\Break}{Break} 
          \DMO{\zero}{zero} 
          \DMO{\ob}{Ob}
          \DMO{\gr}{Gr} 
          \DMO{\Gr}{Gr} 
          \DMO{\cl}{Cl} 
          \DMO{\grlag}{GrLag}
          \DMO{\Pin}{Pin}
          \DMO{\Graph}{Graph}
          \DMO{\pin}{Pin}
          \DMO{\gap}{Gap}
          \DMO{\Ex}{Ex}
          \DMO{\id}{id}
          \DMO{\End}{End}
          \DMO{\sym}{Sym}
          \DMO{\aut}{Aut}
          \DMO{\Aut}{Aut}
          \DMO{\haut}{hAut}
          \DMO{\hAut}{hAut}
          \DMO{\DK}{DK} 
          \DMO{\poly}{poly} 
          \DMO{\diff}{Diff}
          \DMO{\coll}{coll}
          \DMO{\dist}{dist} 
          \DMO{\coker}{coker} 
          \nc{\kernel}{\ker} 
          \DMO{\sspan}{span}
          \DMO{\hocolim}{hocolim}	
          \DMO{\holim}{holim}
          \DMO{\sk}{sk}

          \DMO{\ho}{ho}
          \DMO{\fin}{fin}
          \DMO{\tor}{Tor}
          \DMO{\ext}{Ext}
          \DMO{\ret}{Ret}
          \DMO{\ham}{Ham}
          \DMO{\con}{con}
          \DMO{\leaf}{leaf}
          \DMO{\supp}{supp}
          \DMO{\edge}{edge}
          \DMO{\colim}{colim}
          \DMO{\edges}{edges}
          \DMO{\Image}{image}
          \DMO{\roots}{roots}
          \DMO{\height}{height}
          \DMO{\finmod}{FinMod}
          \DMO{\leaves}{leaves}
          \DMO{\planar}{planar}
          \DMO{\vertices}{vertices}

          \nc{\lagg}{\lag^{\cG}}
          \nc{\iso}{\mathsf{Iso}}
          \nc{\Set}{\mathsf{Set}}
          \nc{\Ass}{\mathsf{ \bf Ass}}
          \nc{\Mod}{\mathsf{Mod}}
          \nc{\modules}{\mathsf{Mod}}
          \nc{\sset}{\mathsf{sSet}}
          \nc{\liou}{\mathsf{Liou}}
          \nc{\poset}{\mathsf{Poset}}
          \nc{\trno}{T^*\RR^n_{\geq 0}}
          \nc{\spectra}{\mathsf{Spectra}}
          \nc{\tensorfin}{\tensor^{\fin}}
          \nc{\lagptg}{\lag_{pt,pt}^{\cG}}
          \nc{\Fin}{\mathcal{F}\mathsf{in}}
          \nc{\lagnl}{\lag_{N,\Lambda}}
          \nc{\lagmlg}{\lag_{M,\Lambda}^{\cG}}
          \nc{\lagsplit}{\lag^{\mathsf{split}}}
          \nc{\lagktimes}{(\lag^{\dd k})^\times}
          \nc{\lagplanar}{\lag^{\times,\planar}}
          \nc{\Cont}{\text{\rm Cont}}
          \nc{\Ham}{\text{\rm Ham}}
          \nc{\Dev}{\text{\rm Dev}}
          \nc{\Lin}{\text{\rm Lin}}
          \nc{\Int}{\text{\rm Int}}
          \nc{\Hom}{\text{\rm Hom}}
          \nc{\Chord}{\text{\rm Chord}}
          \nc{\nbhd}{\mathcal{N}\text{\rm{bhd}}}

          \nc{\onef}{1_{\fukaya}}
          \nc{\smsh}{\wedge}
          \nc{\un}{\underline}
          \nc{\xto}{\xrightarrow}
          \nc{\xra}{\xto}
          \nc{\tensor}{\otimes}
          \nc{\del}{\partial}
          \nc{\dd}{\diamond}
          \nc{\tri}{\triangle}
          \nc{\bb}{\Box}
          \nc{\into}{\hookrightarrow}
          \nc{\onto}{\twoheadrightarrow}
          \nc{\contains}{\supset}
          \nc{\transverse}{\pitchfork}
          \nc{\uncirc}{\underline{\circ}}

          \nc{\Jbar}{\overline{J}}
          \nc{\Fbar}{\overline{F}}
          \nc{\delbar}{\overline{\del}}
          \nc{\thetabar}{\overline{\theta}}
          \nc{\omegabar}{\overline{\omega}}
          \nc{\Liou}{\text{\rm Liou}}

          \nc{\Yhat}{\widehat{Y}}
          \nc{\Xhat}{\widehat{X}}

          \nc{\trbar}{\overline{T^*\RR}}
          \nc{\tr}{T^*\RR}
          \nc{\tsa}{Ts\cA}
          \nc{\tsb}{Ts\cB}
          \nc{\cmbar}{\overline{\cM}}
          \nc{\crbar}{\overline{\cR}}

          \nc{\vece}{ {\vec \epsilon}}	
          \nc{\vecd}{ {\vec \delta}}
          \nc{\ov}{\overline}
          \DMO{\op}{op}
          \nc{\opp}{ ^{\op}}

          \nc{\hiro}{\textcolor{blue}}
          \nc{\YG}{\textcolor{orange}}

          \nc{\eqn}{\begin{equation}}
          \nc{\eqnn}{\begin{equation}\nonumber}
          \nc{\eqnd}{\end{equation}}
          \nc{\enum}{\begin{enumerate}}
          \nc{\enumd}{\end{enumerate}}
          \nc{\beastar}{\begin{eqnarray*}}
          \nc{\eeastar}{\end{eqnarray*}}

          \def\cA{\mathcal A}\def\cB{\mathcal B}\def\cC{\mathcal C}\def\cD{\mathcal D}
          \def\cG{\mathcal G}
          
          \def\cM{\mathcal M}
          \def\cR{\mathcal R}\def\cS{\mathcal S}
          \def\cV{\mathcal V}\def\cW{\mathcal W}

          \def\RR{\mathbb R}
          
          \def\ZZ{\mathbb Z}

          \nc{\Euc}{\mathsf{Euc}}
          \nc{\mfld}{\mathsf{Mfld}}
          \nc{\DTop}{\mathsf{DTop}}
          \nc{\simp}{\mathsf{Simp}}
          \nc{\Ainftycatt}{A_\infty Cat}
          \nc{\dgcatt}{dg Cat}
          \nc{\StableCat}{StableCat}
          \nc{\subdivision}{\mathsf{subdiv}}
          \nc{\Kan}{\mathcal{K}\mathsf{an}}

          \theoremstyle{definition}
          \newtheorem{theorem}{Theorem}[section]
          
          \newtheorem{prop}[theorem]{Proposition}
          \newtheorem{lemma}[theorem]{Lemma}
          \newtheorem{warning}[theorem]{Warning}
          
          \newtheorem{corollary}[theorem]{Corollary}
          \newtheorem{construction}[theorem]{Construction}
          
          \newtheorem{definition}[theorem]{Definition}
          \newtheorem{defn}[theorem]{Definition}
          \newtheorem{notation}[theorem]{Notation}
          \newtheorem{example}[theorem]{Example}
          
          \newtheorem{assumption}[theorem]{Assumption}
          
          \newtheorem{remark}[theorem]{Remark}

\title{Smooth constructions of homotopy-coherent actions}

\author[$\dagger$]{Yong-Geun Oh}
\author[$\star$]{Hiro Lee Tanaka}

\affil[$\dagger$]{Center for Geometry and Physics (IBS), Pohang, Korea \&
Department of Mathematics, POSTECH, Pohang Korea.}
\affil[$\star$]{Department of Mathematics, Texas State University}

\begin{document}

\maketitle

\begin{abstract}
We prove that, for nice classes of infinite-dimensional smooth groups $G$,  natural constructions in smooth topology and symplectic topology yield homotopically coherent group actions of $G$. This yields a bridge between infinite-dimensional smooth groups and homotopy theory.

The result relies on two computations: One showing that the diffeological homotopy groups of the Milnor classifying space $BG$ are naturally equivalent to the (continuous) homotopy groups, and a second showing that a particular strict category localizes to yield the homotopy type of $BG$. 

We then prove a result in symplectic geometry: These methods are applicable to the group of Liouville automorphisms of a Liouville sector. The present work is written with an eye toward \cite{oh-tanaka-actions}, where our constructions show that higher homotopy groups of symplectic automorphism groups map to Fukaya-categorical invariants, and where we prove a conjecture of Teleman from the 2014 ICM in the Liouville and monotone settings.

\end{abstract}

\tableofcontents

\section{Introduction}
Suppose $G$ is an automorphism group of a smooth object. For example, $G$ could be the group of diffeomorphisms of a compact smooth manifold $Q$, or the group of Liouville automorphisms of a Liouville sector (a particularly well-behaved exact symplectic manifold, possibly with boundary). It is a classical---and, in most examples of $G$, unsolved---problem to tractably chracterize the homotopy type of $G$. 

One of the most fruitful ways to study a group is to construct actions $G \circlearrowright X$; in studying the homotopy type of $G$, we are interested in constructing homotopically coherent actions of $G$ (so that the map $G \to \aut(X)$ is not necessarily a group homomorphism, but can be promoted to a map of group-like $E_1$-spaces). In this paper we describe a general method for constructing homotopically coherent actions of $G$. 

The philosophy we employ has classical roots: Suppose that one has an invariant of principle $G$-bundles (over paracompact, smooth manifolds) that respects pullback in some controlled way; such an assignment defines a sheaf on the stack $\mathsf{B}G$, and hence an action of $G$ on the object assigned to a point. However, it is common in smooth geometry to need additional choices to define invariants, and these choices may need to satisfy some form of transversality. To create a homotopically coherent action of $G$ on an object may hence require tedious (because one has to carefully make compatible choices for every $G$-bundle) or impossible (because transversality may not be achievable with  constraints imposed by compatibility) constructions. 

Our present work yields one way to make a ``minimal'' amount of choices to both reduce the tedium and make constructions possible.

To that end, we let $G$ be a well-behaved, possibly infinite-dimensional smooth group. By ``well-behaved,'' we mean that the diffeological homotopy groups are naturally equivalent to the continuous homotopy groups. (See Assumption~\ref{assumption on G} for details and examples.) We let $\widehat{BG}$ be the Milnor classifying space, considered as a diffeological space (see Notation~\ref{notation. hats}), and define a category
	$\simp(\widehat{BG})$
as follows:

Define the extended smooth simplex to be $\widehat{|\Delta_e^n|} = \{x_0 + \ldots + x_n = 1\} \subset \RR^{n+1}$. These are abstractly diffeomorphic to $\RR^n$, and smooth simplices of different dimensions admit smooth maps between them induced by the usual simplicial maps. An object of $\simp(\widehat{BG})$ is a smooth simplex $j: \widehat{|\Delta_e^n|} \to \widehat{BG}$, and a morphism from $j$ to $j'$ is the data of a commutative diagram
	\eqnn
	\xymatrix{
	\widehat{|\Delta_e^n|} \ar[rr] \ar[dr]_j && \widehat{|\Delta_e^{n'}|} \ar[dl]^{j'} \\
	& \widehat{BG}&
	}
	\eqnd
where the horizontal arrow is a simplicial map induced by an order-preserving injection $[n] \to [n']$. Our main result is as follows:

\begin{theorem}\label{main theorem}
Let $G$ be a topological group equipped with a diffeological structure  $\widehat{G}$ satisfying Assumption~\ref{assumption on G} below. 
The $\infty$-categorical localization of $\simp(\widehat{BG})$ (along all its morphisms) is homotopy equivalent to the singular simplicial set $\sing(BG)$ of the Milnor classifying space $BG$.
\end{theorem}

The key points are that $\simp(\widehat{BG})$ is a category in the usual sense, and encodes {\em smooth} data; meanwhile, the theorem tells us that its localization is a rich $\infty$-category homotopy equivalent to (the singular simplicial set of) the space $BG$, encoding purely {\em homotopical} data. As a corollary, we obtain:

\begin{corollary}\label{cor.F}
Let $F: \simp(\widehat{BG}) \to \cC$ be a functor to an $\infty$-category $\cC$ for which every morphism in $\simp(\widehat{BG})$ is sent to an equivalence. Then $F$ induces a functor of $\infty$-categories
	\eqnn
	\sing(BG) \to \cC.
	\eqnd
In particular, if $x_0$ is any vertex of $BG$, $X=F(x_0)$ is an object of $\cC$ equipped with a homotopy coherent action of the $A_\infty$-space $G$. In other words, we obtain a map
	\eqnn
	\sing(G) \to \aut_{\cC}(X)
	\eqnd
of group-like $E_1$-objects. In particular, for every $n {\geq 0}$, we obtain group homomorphisms $\pi_n G \to \pi_n\aut_{\cC}(X)$.
\end{corollary}

\begin{remark}
Because the $\infty$-categories of simplicial sets and of topological spaces are equivalent, the $E_1$-map of simplicial sets $\sing(G) \to \aut_{\cC}(X)$ yields an $E_1$-map of spaces 
	$G \to |\aut_{\cC}(X)|,$
where $|\aut_{\cC}(X)|$ is the geometric realization.
\end{remark}

\begin{remark}\label{remark. making F}
Thus, to produce actions of $G$, it suffices to produce functors $F$ as in the corollary. Such functors are often natural to produce, so let us explain how to construct such an $F$. 

We have chosen $BG$ to be modeled as the Milnor classifying space; this topological space can be given a diffeological space structure $\widehat{BG}$, and equipped with an enumerable diffeological principle $G$-bundle $\widehat{EG} \to \widehat{BG}$ (we recall this in Theorem~\ref{theorem. BG facts} below). 

As a result, any smooth map $j: D \to \widehat{BG}$ pulls back a diffeological principle $G$-bundle. If $G$ acts smoothly on a smooth manifold---call it $Q$---and if $D$ is a smooth paracompact manifold, the associated $Q$-fiber bundle over $D$ is now a smooth bundle in the classical (not just diffeological) sense. The theorem allows us to consider only those $D$ among the simplest of smooth manifolds---the extended $n$-dimensional simplices $\widehat{|\Delta_e^n|}$.
\end{remark}

Thus, to construct $F$, one need only create
\enum
\item An invariant $F(j) \in \cC$ for every $j:\widehat{|\Delta_e^n|} \to \widehat{BG}$. Pulling back the associated $Q$-bundle, this reduces to an invariant of the $Q$-bundle over $\widehat{|\Delta_e^n|}$ determined by $j$.
\item For every simplicial inclusion $\widehat{|\Delta_e^n|} \to \widehat{|\Delta_e^{n'}|}$, an induced map of invariants $F(j) \to F(j')$ which is an equivalence in $\cC$. These maps must respect composition.
\enumd
We note that such a construction---in many situations---is far easier than constructing a sheaf on the stack $\mathsf{B}G$. Not only are there fewer base spaces $D$ to worry about, the set of base spaces we consider is ordered (by dimension). Thus one can first choose invariants freely when $D = |\Delta^0|$, then choose invariants for $j$ when $D = |\Delta^k_e|$ in a way compatible with the inductively determined data on the boundary $j|_{\del |\Delta^k_e|}$. Moreover, there are only {\em finitely} many $j$ that map to $j'$ by definition of the category $\simp(\widehat{BG})$; so it is often easy to construct $F$ and the maps $F(j) \to F(j')$ in a way respecting the transversality requirements at hand. Transversality is harder to achieve if one must construct data for all smooth manifolds mapping to $\widehat{BG}$, and for all smooth maps between such smooth manifolds. Indeed, Corollary~\ref{cor.F} allows us to avoid several analytic annoyances in~\cite{oh-tanaka-actions}.

\begin{remark}
A $Q$-bundle over $\widehat{|\Delta_e^n|}$ is necessarily trivializable, but we note that the concrete model of $EG \to BG$ and the concrete nature of the associated bundle construction allows us to write down this $Q$-bundle quite explicitly. Moreover, choosing a trivialization would defeat the purpose---one cannot compatibly choose trivializations across all $j$ unless the bundle $EG \to BG$ itself is trivial (and the bundle is never trivial for interesting $G$).
\end{remark}

%
%
%

The proof of Theorem~\ref{main theorem} requires two ingredients. The first is purely infinity-categorical: We show that the localization of $\simp(\widehat{BG})$ is weakly homotopy equivalent to a simplicial set $\sing^{C^\infty}(\widehat{BG})$ given by the diffeologically {\em smooth} simplices of $\widehat{BG}$ (Lemma~\ref{lemma. kan completion is smooth sing}).

The second ingredient is to show that the passage between the diffeological world and the continuous world loses no homotopical information (Theorem~\ref{theorem. smooth and continuous homotopy groups BG} below).
Though we label \ref{theorem. smooth and continuous homotopy groups BG} as a theorem, the proof methods are rather elementary. We might file the theorem under ``things that should have been proven already.'' (In fact, during the refereeing process of the present work, we were made aware of the (extensive!) work in~\cite{kihara-smooth-infinite}, which appeared within three months of the arXiv appearance of the present results~\cite[Version 1]{oh-tanaka-actions}. We refer the reader to Section~\ref{section. other works} for more.) 

Finally,  Theorem~\ref{theorem. smooth and continuous homotopy groups BG} requires the hypotheses of Assumption~\ref{assumption on G}, so we exhibit two examples of infinite-dimensional groups satisfying this assumption. An obvious test case is the group $\diff(Q)$ of diffeomorphisms of a  compact smooth manifold $Q$, and we prove that $\diff(Q)$ does satisfy this assumption in  Section~\ref{section. diff Q}. 

A less trivial example is the group $\Liouauto(M)$ of Liouville automorphisms of a Liouville sector. We prove that $\Liouauto(M)$ satisfies Assumption~\ref{assumption on G} in Section~\ref{section. liouville}; this result concerning Liouville automorphisms requires some tools, notably Theorem~\ref{thm. Liou smooth approx}, and does not appear to be in previous literature.

\subsection{Application}\label{section.application}
%
%

A construction of an invariant as outlined in Remark~\ref{remark. making F} is carried out in~\cite{oh-tanaka-actions} to produce actions on wrapped Fukaya categories, and on categories of local systems. Namely, let $G=\Liouauto(M)$ be the automorphism group of a Liouville sector $M$ (see Section~\ref{section. liouville}). By pullback, any object of $\simp(B\widehat{\Liouauto(M)})$ defines a Liouville bundle over an $n$-simplex with fiber $M$. One can construct a version of the wrapped Fukaya category for this fiber bundle, and show that morphisms in $\simp(B\widehat{\Liouauto(M)})$ induce equivalences of these wrapped Fukaya categories.

In this way, Corollary~\ref{cor.F} produces a representation of $\Liouauto(M)$ on the wrapped Fukaya category of $M$; this proves a conjecture of Teleman from the 2014 ICM in the Liouville setting~\cite{teleman-icm}. In fact, the same methods apply when $M$ is compact monotone and $\Liouauto(M)$ is replaced by the group of Hamiltonian automorphisms---this is essentially an $\infty$-categorical riff of constructions already contained in~\cite{savelyev}.

When $G = \diff(Q)$, we construct a homotopically coherent action of $\diff(Q)$ on the category of local systems on $Q$ using similar methodology. Further, when $M = T^*Q$, the common framework for these constructions (i.e., because we use Corollary~\ref{cor.F} for both constructions) allows us to show that the $\Liouauto(T^*Q)$ action on the wrapped Fukaya category of $T^*Q$ extends the $\diff(Q)$ action on local systems of $Q$. (The category of local systems and the wrapped Fukaya category are equivalent by a theorem of Abouzaid.) 

We provide details and applications of this construction in~\cite{oh-tanaka-liouville-bundles} and~\cite{oh-tanaka-actions} .

\begin{remark}
We emphasize that we are able to choose coherent data to construct Fukaya categories as above precisely because of the advantages mentioned in Remark~\ref{remark. making F}.
 
The ability to consider {\em smooth} bundles---i.e., induced by smooth maps $\widehat{|\Delta_e^n|} \to \widehat{B\Liouauto(M)}$, as opposed to continuous maps $|\Delta^n| \to B\Liouauto(M)$---is necessary, as the construction must utilize tools of the smooth world (transversality, connections, et cetera). The ability to make choices for objects of $\simp(\widehat{BG})$---as opposed to the entire category of all smooth manifolds mapping to $\widehat{BG}$, and as opposed to a category with {\em all} smooth maps between simplices---makes it possible to choose compatible and transversal choices.
\end{remark}

\subsection{Relation to other works}\label{section. other works}
For ease of reference, and for this section only, we will utilize notation from the work of Kihara~\cite{kihara-smooth-infinite}. We warn that our math calligraphic font differs from that in Kihara's works. Also, because the discussion is slightly technical, this section is not self-contained; we leave the reader to the rest of this paper, and to~\cite{kihara-smooth-infinite}, for precise definitions. 

There is an important conventional difference between \cite{kihara-smooth-infinite} and \cite{christensen-wu-homotopy-theory}, which bleeds into our work. In the work of Christensen-Wu, the smooth singular complex---which we denote as $\sing^{C^\infty}$---is defined by mapping extended, smooth simplices $|\Delta^k_e|$ into a diffeological space. Kihara's works  instead use maps from the standard (compact) simplex $|\Delta^k|$, but with a very well-chosen diffeology. (It is not the diffeology of $|\Delta^k|$ as a subset of $|\Delta^k_e|$; see 1.2 of~\cite{kihara-model}.) The key property is that horns become diffeological deformation retracts of the $k$-simplex using Kihara's diffeology.

\begin{remark}
Kihara's diffeology is quite natural---first, we note that for any $0 \leq i \leq k$, the cone on the $i$th face of a $k$-simplex is naturally isomorphic to the $k$-simplex. Deleting the cone point, one witnesses a continuous embedding $|\Delta^{k-1}| \times [0,1) \into |\Delta^k|$, and Kihara inductively demands that each of these maps is smooth (after giving $|\Delta^1|$ and $|\Delta^0|$ the standard diffeologies).
\end{remark}

With this said, Kihara~\cite{kihara-quillen-equivalences} shows that the three categories  $\cS$ of simplicial sets, $\cD$ of diffeological spaces , and  $\cC^0$ of arc-generated\footnote{This means that a subset $U$ is open if and only if for every continuous map from $\RR$, the pre-image of $U$ is open.} topological spaces, admit Quillen adjunctions among them as follows: 
\enum
	\item There is a functor $\cD \to \cC^0$ sending a diffeological space to its underlying set equipped with the $D$-topology. This is denoted by $X \mapsto \widetilde X$, and admits a right adjoint denoted by $R$. $R$ sends a space $Y$ to a ``boring'' diffeology where any continuous map is declared smooth. 
	\item There is a functor $S^{\cD}$ sending a diffeological space to the simplicial set whose $k$-simplices consist of smooth maps from $\widehat{|\Delta^k|}$. Here, we are endowing $|\Delta^k|$ with the diffeology introduced by Kihara. Importantly, $S^{\cD}$ is always a Kan complex (Lemma~9.4 of~\cite{kihara-model}), while Christensen-Wu's construction---which we denote $\sing^{C^\infty}$ in our work---may not be (Warning~\ref{warning. homogeneity and fibrancy}). At the same time, the simplicial homotopy groups of $S^{\cD}$ are naturally isomorphic to the diffeological homotopy groups of the original diffeological space (Theorem~1.4 of~\cite{kihara-model}). The functor $S^{\cD}$ admits a left adjoint $|-|_{\cD}$, called realization, which takes a simplicial set to the usual geometric realization, but constructed as a colimit in diffeological spaces (as opposed to topological spaces). As usual, $S^{\cD}$ is right adjoint to $|-|_{\cD}$.
	\item The composition of these is the usual ``singular complex / realization'' adjunction between arc-generated topological spaces and simplicial sets. In the present work, we denote the singular complex by $\sing$.
\enumd
These are all Quillen equivalences. (See \cite{haraguchi-shimakawa} and~\cite{kihara-quillen-equivalences}, though the original version of \cite{haraguchi-shimakawa} seems to have a gap due to the issues in their chosen diffeology for simplices.) 

Following Kihara's notation, let us denote by $\cV_{\cD}$ the class of those diffeological spaces for which $S^{\cD}(-) \to \sing( \widetilde{-})$ is a homotopy equivalence of simplicial sets. (Equivalently, the natural map from the diffeological homotopy groups to the continuous homotopy groups of the $D$-topologized space is an isomorphism.) In the work~\cite{kihara-smooth-infinite}, Kihara identifies many diffeological spaces belonging to $\cV_{\cD}$. In particular, Corollary~11.22 of~\cite{kihara-smooth-infinite} shows that when $Q$ is a compact smooth manifold, $\diff(Q)$ is in this class.\footnote{In Kihara's notation, they show that $\diff(Q)$ is in a class called $\cW_{\cD 0}$, which is contained in $\cV_{\cD}$ by Corollary~1.6 and Remark~9.21 of ibid.}

Denote by $\widehat{|\Delta^k|}_{\text{sub}}$ the subset diffeology of $|\Delta^k| \subset \widehat{|\Delta^k_e|}$. Then the identity map of $|\Delta^k|$ is a smooth map $\widehat{|\Delta^k|} \to \widehat{|\Delta^k|}_{\text{sub}}$ when the domain is given Kihara's diffeology (Lemma~3.1 of~\cite{kihara-model}). Since the inclusion  $\widehat{|\Delta^k|}_{\text{sub}} \to \widehat{|\Delta^k_e|}$ is smooth by definition, restriction induces a map
	$
	\sing^{C^\infty} \to S^{\cD}
	$
and hence a factorization
	\eqnn
	\sing^{C^\infty}(-) \to S^{\cD}(-) \to \sing(\widetilde{-})
	\eqnd
for any input diffeological space. 

Our work applies to groups for which this composition is a weak homotopy equivalence (Assumption~\ref{assumption on G}). The first map $\sing^{C^\infty} \to S^{\cD}$ is expected to be a weak homotopy equivalence for any diffeological space (Remark~A.5 of~\cite{kihara-model}); in fact, we suspect that for the class of diffeological spaces identified in Theorem 11.20 of \cite{kihara-smooth-infinite}, every map---in particular, the last map---in this composition is a weak homotopy equivalence. While such a statement would be satisfying, in our examples we instead study the total composite and prove that the total composite is a weak homotopy equivalence without proving anything about the factorizing maps, nor invoking Kihara's results (of which we were unaware until the first incarnation of our work already appeared). Admittedly, a far more aesthetically pleasing avenue may be to get rid of $\sing^{C^\infty}$ altogether---indeed, in all our examples, the reader will readily see that the proofs in the present work explicitly demonstrate the homotopy equivalence of the arrow $S^{\cD}(-) \to \sing(\widetilde{-})$.

\subsection{Notation}
Let us set some notation and conventions. Throughout this work, we will assume that the reader is familiar with simplicial and $\infty$-categorical constructions, including the notions of coCartesian fibrations and Cartesian fibrations. For background on (co)Cartesian fibrations, we refer the reader to Section~2.9 of~\cite{oh-tanaka-localizations}, Section~4 of~\cite[Version 1]{tanaka-paracyclic}, and Section~3.2 of~\cite{htt}.

Because it will be important to distinguish between a topological space and a choice of smooth structure on it, we hereby enact the following:

\begin{notation}[Smooth players wear hats]\label{notation. hats}
We will refer to an object with smooth structure by $\widehat{B}$ (i.e., by making the symbol wear a hat). $B$ will often denote an underlying set, or space, associated to $\widehat{B}$. For example, $\widehat{BG}$ is a diffeological space, while $BG$ is the Milnor classifying space associated to the topological group $G$. (These have the same underlying set.)
\end{notation}

\begin{notation}[The nerve $N(\cC)$]\label{notation. nerve}
As usual, if $\cC$ is a category, the nerve of $\cC$ is a simplicial set whose $k$-simplices consist of commutative diagrams in $\cC$ in the shape of a $k$-simplex. We let $N(\cC)$ denote the nerve.
\end{notation}

\begin{notation}[The combinatorial $n$-simplex]
As usual we let $\Delta^n$ denote the simplicial set represented by the poset $[n]$.

In contrast, $|\Delta^n|$ refers to the topological space given by the standard $n$-simplex. $\widehat{|\Delta_e^n|}$ is the extended smooth simplex, see Definition~\ref{defn. extended simplices}. 
\end{notation}

\subsection{Acknowledgments}
We would like to thank
Dan Christensen for helpful comments and pointing out mistakes in a previous draft; we also thank the anonymous reviewers whose comments have significantly improved the exposition of the work. 
The first author is supported by the IBS project IBS-R003-D1.
The second author was supported by
IBS-CGP in Pohang, Korea and
the Isaac Newton Institute in Cambridge, England,
during the preparation of this work. This material is also based upon work supported by the National Science Foundation under Grant No. DMS-1440140 while the second author was in residence at the Mathematical Sciences Research Institute in Berkeley, California, during the Spring 2019 semester.

\section{Reminders on diffeological spaces}

We collect various results, many of which are due to the papers of Christensen-Sinnamon-Wu, Christensen-Wu, and Magnot-Watts~\cite{christensen-sinnamon-wu, christensen-wu-homotopy-theory, christensen-wu-smooth-classifying-spaces, magnot-watts}.

\begin{notation}[$\mfld$ and $\Euc$.]
Let $\mfld$ denote the category of smooth manifolds---its objects are smooth manifolds, and morphisms are smooth maps. We let $\Euc \subset \mfld$ denote the full subcategory of those manifolds that are diffeomorphic to an open subset of $\RR^n$ for some $n$.
\end{notation}

When defining a geometric object, one can take a Lawvere-type approach to define functions on that object, or one can take a functor-of-points approach to define functions into that object. A diffeological space is defined by the latter approach: We will often define a diffeological space by beginning with the data of a set $X$, and then for all $U \in \ob \Euc$, specifying which functions $U \to X$ are ``smooth.'' This defines a functor $\widehat{X}: \Euc^{\op} \to \sets$ as in the following definition:

\begin{defn}[Diffeological space]\label{defn. diffeological space}
Fix a functor $\widehat{X}: \Euc^{\op} \to \sets$ (i.e., a presheaf of sets on $\Euc$). We say that $\widehat{X}$ is a {\em diffeological space} if the following two conditions hold:
\enum
\item\label{item. functions determined on points} For any $U \in \Euc$, the function
	\eqnn
	\widehat{X}(U) \to \hom_{\sets}(\hom_{\Euc}(\RR^0, U), \widehat{X}(\RR^0))
	\eqnd
	is an injection. (That is, functions are determined by their values on points of $U$.)
\item $\widehat{X}$ is a sheaf (with the usual notion of open cover on smooth manifolds).
\enumd
A map of diffeological spaces---also known as a smooth map of diffeological spaces---is a map of presheaves.
\end{defn}

\begin{remark}
The map in~\ref{item. functions determined on points}. is induced by the structure map for presheaves
	\eqnn
	\widehat{X}(U) \times \hom_{\Euc}(\RR^0, U) \to \widehat{X}(\RR^0).
	\eqnd
\end{remark}

\begin{notation}[Underlying set $X$]
Let $\widehat{X}$ be a diffeological space. Then we say that $\widehat{X}(\RR^0)$ is the {\em underlying set} of $\widehat{X}$, and we denote it $X$. Note that by \ref{item. functions determined on points}., every element of $\widehat{X}(U)$ determines a function $f: U \to X$. If $f$ is in the image of the map in \ref{item. functions determined on points}., we say that $f$ is a {\em smooth map} from $U$ to $X$. (In the literature, this is also called a {\em plot}.)

Unwinding the definitions, we thus see that a diffeological space is equivalent to the  data of a set $X$, and for every $U \in \ob \Euc$, a subset $\widehat{X}(U) \subset \hom_{\sets}(U,X)$, subject to the following properties:
\begin{itemize}
	\item $\widehat{X}(U)$ contains all the constant maps.
	\item If $U \to U'$ is smooth, then the composite $U \to U' \to X$ is in $\widehat{X}(U)$ whenever the function $U' \to X$ is in $\widehat{X}(U')$.
	\item If there is an open cover $\{U_i\}$ of $U$ such that the function $U \to X$ factors as $U_i \to X$, and each $U_i \to X$ is in $\widehat{X}(U_i)$, then the function $U \to X$ is in $\widehat{X}(U)$. 	
\end{itemize}
\end{notation}

\begin{defn}[Smoothness of maps]
Let $\widehat X$ and $\widehat Y$ be diffeological spaces. A function $f: X \to Y$ of underlying sets is called {\em smooth} if it is induced by a map of diffeological spaces. Concretely, $f$ is smooth if for every smooth map $U \to X$, the composite $U \to X \to Y$ is smooth.
\end{defn}

\begin{remark}[$D$-topology]
Let $\widehat{X}$ be a diffeological space and $X$ its underlying set.
One can endow $X$ with the finest topology for which every smooth function $f: U \to X$ determined by $\widehat{X}$ is continuous. This is called the {\em $D$-topology} in the literature.
\end{remark}

\begin{warning}
{\em However,} we will almost never make use of the $D$-topology, and in fact our main example $X = \Liouauto(M)$ will {\em not} be endowed with the $D$-topology. So the reader should not assume that the underlying set $X$ of a diffeological space $\widehat{X}$ is endowed with the $D$-topology.
\end{warning}

\begin{example}[Smooth manifolds]\label{example. smooth manifolds}
Let $X$ be a smooth manifold. Then one can define a diffeological space by declaring $\widehat{X}(U) = \hom_{\mfld}(U,X)$; note that $X$ is indeed the underlying set of $\widehat{X}$ as implied by our notation, and the $D$-topology coincides with the usual one. This construction gives a fully faithful embedding of the category of smooth manifolds into the category of diffeological spaces.
\end{example}

\begin{example}[Subspaces]\label{example. subspace diffeology}
Let $\widehat{X}$ be a diffeological space, and let $A \subset X$ be a subset. Then $A$ determines a subsheaf $\widehat{A} \subset \widehat{X}$ where $\widehat{A}(U)$ consists of all those elements $f: U \to X$ whose image lies in $A$. We call this the subspace diffeology on $A$. Note this is an example where there is ambiguity in the topology of $A$---one could give it the subspace topology with respect to the $D$-topology on $X$, or give it the $D$-topology induced by the diffeological structure $\widehat{A}$. These topologies do not always coincide.
\end{example}

\begin{example}[Function spaces]\label{example. function spaces}
Let $X$ and $Y$ be smooth manifolds, and let $C^\infty(X,Y)$ be the set of smooth functions. We can endow this set with a diffeological space structure by declaring a function $U \to C^\infty(X,Y)$ to be smooth if and only if the adjoint map $U \times X \to Y$ is smooth. In general, the $D$-topology of this set is finer than the compact-open topology\footnote{That is, the topology inherited as a subset of the compact-open topologized collection of continuous maps from $X$ to $Y$}, finer than the weak Whitney topology\footnote{Confusingly, this is sometimes called the compact-open topology in the smooth literature.}, but coarser than the strong Whitney topology~\cite{christensen-sinnamon-wu}.
\end{example}

\begin{remark}\label{remark. diffeological spaces have lims and colims}
The category of diffeological spaces has all limits and colimits; in fact, the functor sending a diffeological space to its underlying set has both left and right adjoints, so the underlying sets of limits and colimits can be understood in the usual way. Moreover, the functor sending a diffeological space to its underlying space (with the $D$-topology) admits a right adjoint, so the $D$-topologized topological space of colimits can be understood in terms of colimits of spaces in the usual way. We also have an explicit description of the colimit diffeology: A function $U \to \colim f$ is smooth if and only if there is an open cover $\{U_i\}$ of $U$, and an object $f(j)$ in the diagram given by $f$, such that the function factors $U_i \to f(j) \to \colim f$ with $U_i \to f(j)$ being smooth.

One can also show that the category of diffeological spaces is Cartesian closed. The hom-objects are precisely the function spaces with the diffeological space structure of Example~\ref{example. function spaces}.

We also remark that these observations follow straightforwardly from the fact that the category of diffeological spaces is equivalent to the category of so-called ``concrete'' sheaves on a site. (Concreteness is precisely equivalent to condition \ref{item. functions determined on points} of Definition~\ref{defn. diffeological space}.)
\end{remark}

\subsection{Diffeological groups}

\begin{defn}\label{defn. diffeological group}
A {\em diffeological group} is a group object in the category of diffeological spaces. Concretely, this is the data of a diffeological space $\widehat{G}$, together with a group structure whose inverse and multiplication operations are smooth.
\end{defn}

\begin{example}\label{example. diff groups and their subgroups}
For any smooth manifold $X$, the diffeomorphism group $\diff(X)$ is a diffeological space. It is in fact a diffeological group (Definition~\ref{defn. diffeological group}). A priori, this diffeology is not induced by the subspace diffeology of Examples~\ref{example. subspace diffeology} and~\ref{example. function spaces}; one must reduce the number of smooth maps to guarantee that the inverse function is smooth. That is, a map $j: U \to \diff(X)$ is declared smooth not only if $(u,x) \mapsto j(u)(x)$ is smooth, but also if $(u,x) \mapsto j(u)^{-1}(x)$ is smooth (as maps $U \times X \to X$). See Example~2.12 of~\cite{christensen-wu-smooth-classifying-spaces}.
\end{example}

However, when $X$ is a smooth manifold, the subset diffeology (Example~\ref{example. subspace diffeology}) is compatible with the group diffeology (Example~\ref{example. diff groups and their subgroups}) of $\diff(X)$:

\begin{prop}\label{prop. diff X has one diffeology}
Let $X$ be a smooth manifold.
The subset diffeology of $\diff(X)$ inherited from $C^\infty(X,X)$ is identical to the diffeology of $\diff(X)$ as a diffeological group.
\end{prop}

\begin{proof}
Let $U$ be an open subset of $\RR^n$.
It suffices to show that if $\phi: U \times X \to X$ is smooth, then the map $\psi: U \times X \to X, (u,x) \mapsto \phi_u^{-1}(x)$ is smooth. 

This can be checked locally on $U \times X$. So fix $(u_0,x_0) \in U \times X$ and choose a Riemannian metric on $X$. Shrink $U$ and $X$ so that $\psi_{u_0}(\phi(U \times X))$ is a subset of $X$ and is a geodesically convex neighborhood of $x_0$. For every $u' \in U$, this exhibits $\psi_{u_0} \circ \phi_{u'} : X \to X$ as the time-1 flow of a vector field. (The flow moves a point $x'$ to the point $\psi_{u_0} \circ \phi_{u'}(x')$ along their unique geodesic.) Because $\phi: U \times X \to X$ is smooth, these vector fields depend smoothly on $u' \in U$; let us call the vector field $v_{u'}$ so that $\flow^{t=1}_{v_{u'}} = \psi_u \circ \phi_{u'}$. Shrinking $U$ again if necessary, and because time-reversal of flows produces set-theoretic inverses, we find that
	\eqnn
	\flow^{t=1}_{-v_{u'}} = (\psi_{u_0} \circ \phi_{u'})^{-1}.
	\eqnd
In other words,
	\eqnn
	\flow^{t=1}_{-v_{u'}} \circ \psi_{u_0} = \phi_{u'}^{-1}
	\eqnd
(shrinking $U$ as necessary).
By smooth dependence of solutions to ODEs, and smoothness of negation $v \mapsto -v$, the lefthand side depends smoothly on $(u',x')$, hence $\phi_{u'}^{-1} = \psi_{u'}$ does, too. This shows that $\psi$ is a smooth map $\psi: U \times X \to X$. 
\end{proof}
%

\subsection{Some homotopy theory of diffeological spaces}

One of the most useful tools in homotopy theory is the ability to convert any topological space into a simplicial set. We recall the analogue of this for diffeological spaces.

\begin{defn}[$|\Delta_e^k|$]\label{defn. extended simplices}
Let $|\Delta_e^k| \subset \RR^{k+1} \cong \hom([k],\RR)$ denote the affine hyperplane defined by the equation $\sum_{i=0}^k t_i = 1$. We refer to $|\Delta_e^k|$ as the {\em extended} $k$-simplex, and consider it a smooth manifold in the obvious way. (It is diffeomorphic to the standard Euclidean space $\RR^k$.)

We will refer to a map $|\Delta_e^k| \to |\Delta_e^{k'}|$ as {\em simplicial} if it is the restriction of the linear map $\RR^{[k]} \to \RR^{[k']}$ induced by some map of sets $[k] \to [k']$. (This map need not respect order.)
\end{defn}

In the diffeological space setting, we make the following definition:

\begin{definition}[Smooth homotopy groups]\label{defn. smooth pi_n}
The $n$th smooth homotopy group
	\eqnn
	\pi_n^{C^\infty}(\widehat{X},x_0)
	\eqnd
for $x_0 \in X$ is the group of smooth homotopy classes of maps
	$	
	f: \widehat{|\Delta_e^n|} \to \widehat{X}	
	$
satisfying the condition that $f(y) = x_0$ for any $y \in \widehat{|\Delta_e^n|}$ for which $y$ has some coordinate equal to zero. The homotopy classes of maps are taken relative to the subset of those $y \in \widehat{|\Delta_e^n|}$ with at least one coordinate equal to zero.
\end{definition}

\begin{remark}
The above model for $\pi_n^{C^\infty}$ is equivalent to many others. (For example, one could take $\pi_n^{C^\infty}$ to be defined as smooth homotopy classes of smooth, pointed maps from the standard smooth $n$-sphere.) See Theorem~3.2 of~\cite{christensen-wu-homotopy-theory}, where $\pi_n^{C^\infty}$ is written as $\pi_n^{D}$.
\end{remark}

In the usual homotopy theory of topological spaces, we can compare two different notions of homotopy groups. The usual notion $\pi_n$ is defined by based homotopy classes of continuous maps $S^n \to X$, and the combinatorial definition is defined by classes of maps from $\Delta^n$ to the simplicial set $\sing(X)$. Let us explain  the analogue of $\sing$ in the diffeological setting.

\begin{remark}
Since any smooth manifold is a diffeological space (Example~\ref{example. smooth manifolds}), the assignment $[k] \mapsto \widehat{|\Delta_e^k|}$ defines a cosimplicial object in the category of diffeological spaces.
\end{remark}

\begin{notation}[$\sing^{C^\infty}(\widehat{X})$]
Let $\widehat{X}$ be a diffeological space. We let $\sing^{C^\infty}(\widehat{X})$ denote the simplicial set
	\eqnn
	\sing^{C^\infty}(\widehat{X}): \Delta^{\op} \to \sets,
	\qquad
	[k] \mapsto \hom_{C^\infty}(\widehat{|\Delta_e^k|}, \widehat{X})
	\eqnd
whose $k$-simplices consist of maps (of diffeological spaces) from extended $k$-simplices to $\widehat{X}$.
\end{notation}

\begin{remark}
In~\cite{christensen-wu-homotopy-theory}, the notation $S^D(X)$ is used to denote what we write as $\sing^{C^\infty}(\widehat{X})$. Also in loc. cit., the $\widehat{X}$ notation is not used to distinguish a diffeological space from its underlying set. We warn the reader that the notation $S^D$ in~\cite{christensen-wu-homotopy-theory} does not mean the same thing as the notation $S^{\cD}$ in Kihara's works. (See Section~\ref{section. other works}.)
\end{remark}

\begin{notation}[$\sing(X)$]
Let $X$ be a topological space. We let $\sing(X)$ denote the usual simplicial set of {\em continuous} simplices $|\Delta^k| \to X$. 
\end{notation}

\begin{remark}
Let $\widehat{X}$ be a diffeological space, and endow the underlying set $X$ with a topology (not necessarily the $D$-topology). We have a a natural restriction operation
	\eqn\label{eqn. restriction of simplices}
		\left(f: \widehat{|\Delta_e^k|} \to \widehat{X}\right)
		\mapsto
		\left( f|_{|\Delta^k|}: |\Delta^k| \to X \right)
	\eqnd
which sends the collection of set-theoretic functions $\widehat{|\Delta_e^k|} \to \widehat{X}$  to the collection of set-theoretic functions $|\Delta^k| \to X$. If $X$ is given the $D$-topology, this restricts to a map $\sing^{C^\infty}(\widehat{X}) \to \sing(X)$.
\end{remark}

\begin{remark}\label{remark. D top maps to weak}
Let $X$ and $Y$ be smooth manifolds, and let $C^\infty(X,Y)$ be the set of smooth functions. 
Let us recall from Example~\ref{example. function spaces}
 that the $D$-topology of $C^\infty(X,Y)$ contains the weak topology---in fact, it is the smallest arc-generated topology of $C^\infty(X,Y)$ containing the $D$-topology. As a result, we have a continuous map (given as the identity on sets) $C^\infty(X,Y)_D \to C^\infty(X,Y)_{\weak}$ and a map of simplicial sets
	\eqn\label{eqn. sing Coo to sing}
	\sing^{C^\infty}(\widehat{C^\infty(X,Y)})
	\to \sing(C^\infty(X,Y)_{\weak})
	\eqnd
given by restriction \eqref{eqn. restriction of simplices}.
\end{remark}

\begin{warning}\label{warning. homogeneity and fibrancy}
The simplicial set $\sing^{C^\infty}(\widehat{X})$ need not be a Kan complex. (In contrast: For a space $X$, $\sing(X)$ is always a Kan complex.)

In general, it seems that a ``homogeneity'' property is needed to conclude that $\sing^{C^\infty}(\widehat{X})$ is a Kan complex. For example, when $\widehat{X}$ is the diffeological space associated to a smooth manifold, $\sing^{C^\infty}(\widehat{X})$ is a Kan complex (Corollary~4.36 of~\cite{christensen-wu-homotopy-theory}). But when $\widehat{X}$ is associated to a smooth manifold with non-empty boundary, this is no longer true (Corollary~4.47 of ibid.).
\end{warning}

Following the theme of ``homogeneity implies Kan,'' we have the following.

\begin{prop}[Proposition~4.30 of~\cite{christensen-wu-homotopy-theory}.]\label{prop. Sing for diff group is Kan}
For any diffeological group $\widehat{G}$ (see Definition~\ref{defn. diffeological group}), $\sing^{C^\infty}(\widehat{G})$ is a Kan complex.
\end{prop}

\begin{remark}\label{remark. smooth homotopy groups are combinatorial for fibrants}
Now we have two notions of homotopy groups of a diffeological space: One defined as in Definition~\ref{defn. smooth pi_n}, and another using the combinatorial definition of homotopy groups of simplicial sets. 
Let us assume $\widehat{X}$ is a diffeological space for which $\sing^{C^\infty}(\widehat{X})$ is a Kan complex. Then the natural maps
	\eqn\label{eqn. pi_n maps}
	\pi_n^{C^\infty}(\widehat{X},x_0) \to \pi_n(\sing^{C^\infty}(\widehat{X}), x_0).
	\eqnd
are bijections for $\pi_0$, and are isomorphisms for any choice of $x_0 \in X$.
(Theorem~4.11 of~\cite{christensen-wu-homotopy-theory}).
\end{remark}

In particular, this holds for diffeological groups by Proposition~\ref{prop. Sing for diff group is Kan}. Let us record this:

\begin{prop}\label{prop. smooth pi n is combinatorial}
Let $\widehat{G}$ be a diffeological group. Then the map
	$
	\pi_0^{C^\infty}(\widehat{G}) \to \pi_0(\sing^{C^\infty}(\widehat{G}))
	$
is a group isomorphism. Moreover, the map \eqref{eqn. pi_n maps} is an isomorphism for any choice of $x_0 \in G$.
\end{prop}

\begin{remark}\label{remark. irrational torus}
It is not always true that the (continuous) homotopy groups with respect to the D-topology are isomorphic to the smooth homotopy groups. A counterexample is the irrational torus, see Example 3.20 of  \cite{christensen-wu-homotopy-theory}.
\end{remark}

\section{Smooth approximation}

Throughout, we make the following assumption on $G$:

\newenvironment{G-assumption-list}{
	  \renewcommand*{\theenumi}{(A\arabic{enumi})}
	  \renewcommand*{\labelenumi}{(A\arabic{enumi})}
	  \enumerate
	}{
	  \endenumerate
}

\begin{assumption}\label{assumption on G}
Let $G$ be a group. We assume we have chosen a topology and a diffeology on $G$ such that the following hold:

\begin{G-assumption-list}
	\item\label{item. G is group} $G$ is a topological group and a diffeological group.
	\item\label{item. smooth simplices are continuous} Any smooth map $\widehat{|\Delta_e^n|} \to \widehat{G}$ restricts to a continuous map $|\Delta^n| \to G$.
	\item\label{item. smooth pi n are pi n} The induced map of simplicial sets $\sing^{C^\infty}(\widehat{G}) \to \sing(G)$ is a homotopy equivalence.
\end{G-assumption-list}
\end{assumption}

See Sections~\ref{section. diff Q} and~\ref{section. liouville} for two examples.

The aim of the present section is to prove the following basic result; we will shortly include background details for the reader's benefit.

\begin{theorem}\label{theorem. smooth and continuous homotopy groups BG}
Suppose that $G$ is given a diffeology and a topology satisfying Assumption~\ref{assumption on G}. 
Let $\widehat{BG}$ be the Milnor classifying space, equipped with a diffeological space structure as in~\cite{christensen-wu-smooth-classifying-spaces}. Then the inclusion of smooth simplices into continuous simplices
	\eqnn
	\sing^{C^\infty}(\widehat{BG}) \to \sing(BG)
	\eqnd
is a weak homotopy equivalence of simplicial sets. Moreover, the natural maps
	\eqnn
	\pi_n^{C^\infty}(\widehat{BG}) \to \pi_n(BG)
	\eqnd
are isomorphisms.
\end{theorem}

\begin{remark}
The last statement concerns the smoothly defined homotopy groups, which are not a priori equivalent to the combinatorially defined homotopy groups of simplicial sets. An a priori deduction would result from proving that $\sing^{C^\infty}(B\diff(Q))$ is a Kan complex, but we do not investigate whether this simplicial set is Kan in this work.
\end{remark}

\subsection{The classifying space \texorpdfstring{$\widehat{BG}$}{BG-hat}}\label{section. bg}
Now let us recall constructions of classifying spaces for diffeological groups. We follow~\cite{magnot-watts} and ~\cite{christensen-wu-smooth-classifying-spaces}.

Let $\widehat{G}$ be a diffeological group. Then in~\cite{magnot-watts} and~\cite{christensen-wu-smooth-classifying-spaces}, the authors construct two diffeological spaces $\widehat{EG}$ and $\widehat{BG}$, together with a smooth map $\widehat{EG} \to \widehat{BG}$. We employ the diffeology of \cite{christensen-wu-smooth-classifying-spaces}, as the resulting statements are a bit more general.

\begin{remark}
The constructions of $\widehat{EG}$ and $\widehat{BG}$ below are modeled on Milnor's join construction~\cite{milnor-universal-bundles}. One obvious reason to prefer Milnor's construction as opposed to the usual simplicial space construction is that the map $EG \to BG$ need not be locally trivial for the latter; this was pointed out as early as~\cite{segal-classifying-spaces}.
\end{remark}

\begin{construction}[$\widehat{EG}$.]\label{construction.EG}
Let $|\Delta^\omega|$ denote the infinite-dimensional simplex. That is, it is the set of those $(t_i)_{i \in \ZZ_{\geq 0}} \in \bigoplus_{i \geq 0} \RR_{\geq 0}$ for which only finitely many $t_i$ are non-zero, and $\sum t_i = 1$. As a diffeological space, we have that
	\eqnn
	\widehat{|\Delta^\omega|} \cong \colim_{i \geq 0} \widehat{|\Delta^i|}_{\text{sub}}
	\eqnd
where $\widehat{|\Delta^i|}_{\text{sub}}$ is the standard $i$-dimensional simplex, given the subspace diffeology from $\RR^{i+1}$.

\begin{warning}\label{warning. BG with Kihara diffeology}
We have followed Christensen-Wu's construction of $\widehat{EG}$, but in light of Kihara's diffeology for simplices (Section~\ref{section. other works}), it may behoove us to define the diffeology on $|\Delta^\omega|$ as a colimit of Kihara's simplices instead. Certainly the identity map between this paper's $|\Delta^\omega|$ and the colimit of Kihara's simplices is not a diffeological isomorphism. We have not checked whether the diffeological space $\widehat{EG}$ (nor $\widehat{BG}$) remains unchanged regardless of which diffeology one utilizes for simplices.
\end{warning}

Then $|\Delta^\omega| \times \prod_\omega G$ can be given the product diffeology, and we define $\widehat{EG}$ to be the quotient by identifying $(t_i, g_i) \sim (t_i', g_i')$ when the following two conditions hold:
	\enum
	\item $t_i = t_i'$ for all $i$, and
	\item If $t_i = t_i' \neq 0$, then $g_i = g_i'$.
	\enumd
Obviously $\widehat{EG}$ retains the projection map to $|\Delta^\omega|$; its fibers above an element $(t_i)$ can be identified with the product $\prod_{i \text{ s.t. } t_i \neq 0} G$.
\end{construction}

Of course, $\widehat{EG}$ is a diffeological space by virtue of the category of diffeological spaces having all limits and colimits (Remark~\ref{remark. diffeological spaces have lims and colims}). Likewise, the following is a diffeological space:

\begin{construction}[$\widehat{BG}$.]\label{construction.BG}
We let $\widehat{BG}$ be the quotient by the natural action
	\eqnn
	\widehat{EG} \times G \to \widehat{EG}, \qquad
	[(t_i,g_i)]\cdot g = [(t_i,g_i g)].
	\eqnd
\end{construction}

These satisfy a series of properties that we state as a single theorem.

\begin{theorem}\label{theorem. BG facts}
The following are true:
\enum
\item The underlying set of $\widehat{BG}$ is the Milnor construction of the classifying space $BG$~\cite{milnor-universal-bundles}.
\item \label{item. EG contractible} $\widehat{EG}$ is smoothly contractible. (Corollary~5.5 of~\cite{christensen-wu-smooth-classifying-spaces}.)
\item\label{item. EG is a bundle over BG} The map $\widehat{EG} \to \widehat{BG}$ is a diffeological principal $\widehat{G}$-bundle. (Theorem~5.3 of~\cite{christensen-wu-smooth-classifying-spaces}.)
\item For any paracompact diffeological space $\widehat{X}$, pull-back induces a bijection between smooth homotopy classes of maps $\widehat{X} \to \widehat{BG}$, and principal $\widehat{G}$-bundles over $\widehat{X}$. (Theorem~5.10 of~\cite{christensen-wu-smooth-classifying-spaces}.)
\enumd
\end{theorem}

\begin{remark}\label{remark. EG contractible}
That $\widehat{EG}$ is smoothly contractible means that we can find a smooth map $|\Delta^1_e| \times \widehat{EG} \to \widehat{EG}$ which interpolates between a constant map and the identity map. By Lemma~4.10 of~\cite{christensen-wu-homotopy-theory}, this induces a simplicial homotopy between the identity map of $\sing^{C^\infty}(\widehat{EG})$ and a constant map---that is, $\sing^{C^\infty}(\widehat{EG})$ is weakly contractible as a simplicial set, hence weakly homotopy equivalent to a point.
\end{remark}

\begin{remark}\label{remark. bundles are kan fibrations}
Let $\widehat{E} \to \widehat{B}$ be a diffeological bundle, and assume that the fibers are diffeological spaces for whom $\sing^{C^\infty}$ is a Kan complex. Then the induced map $\sing^{C^\infty}(\widehat{E}) \to \sing^{C^\infty}(\widehat{B})$ is a Kan fibration. (See Proposition~4.28 of~\cite{christensen-wu-homotopy-theory}.)

As a result, Theorem~\ref{theorem. BG facts}\eqref{item. EG is a bundle over BG} implies that we have a Kan fibration sequence of simplicial sets
	\eqnn
	\sing^{C^\infty}(\widehat{G}) \to
	\sing^{C^\infty}(\widehat{EG}) \to
	\sing^{C^\infty}(\widehat{BG}).
	\eqnd
\end{remark}

\begin{lemma}\label{lemma. fiber sequence maps}
If $G$ satisfies~\ref{item. smooth simplices are continuous}, one has a commutative diagram of simplicial sets
		\eqnn
		\xymatrix{
		 \sing^{C^\infty}(\widehat{ G }) \ar[r] \ar[d]&  \sing^{C^\infty}(\widehat{E G }) \ar[r] \ar[d]&   \sing^{C^\infty}(\widehat{B G })\ar[d] \\
		 \sing( G ) \ar[r] & \sing(E G ) \ar[r] & \sing(B G )
		}
		\eqnd
where the vertical arrows are the restriction maps $f \mapsto f|_{\Delta^\bullet}$ from \eqref{eqn. restriction of simplices}.
\end{lemma}

\begin{remark}\label{remark. topology on EG and BG}
Since we have specified a topology on $G$ in Assumption~\ref{assumption on G}, the sets $EG$ and $BG$ have induced topologies by performing Constructions~\ref{construction.EG} and \ref{construction.BG} as topological spaces (not diffeological spaces). 
\end{remark}
	
\begin{proof}[Proof of Lemma~\ref{lemma. fiber sequence maps}.]
One need only check that the restriction of $f$ to $|\Delta^k|$ is continuous for every choice of $\widehat{X}$. For $\widehat{X} = \widehat{G}$, this is~\ref{item. smooth simplices are continuous}.
	
	For $\widehat{X} = \widehat{EG}$, the definition of the colimit diffeology (Remark~\ref{remark. diffeological spaces have lims and colims}) guarantees that there is some cover $|\Delta_e^k| = \bigcup_i U_i$ and smooth maps $f_i: U_i \to |\Delta^{n_i}| \times G^{n_i}$ such that $f$ fits into a diagram
		\eqnn
		\xymatrix{
		U \ar[rr]^f  && \widehat{EG} \\
		\bigcup_i U_i \ar[r]^-{\bigcup_i f_i} \ar[u]
			&\bigcup_i |\Delta^{n_i}| \times G^{n_i} \ar[r]
			& |\Delta^\omega| \times \prod_\omega G \ar[u]_{/\sim}
		}
		\eqnd
	where the bottom-right horizontal arrow is the (union of the) obvious inclusion map. (The inclusion map identifies $G^{n_i}$ with the subspace of $\prod_\omega G$ consisting of those $(g_j)$ whose components for $j >n_i$ are all equal to the identity element of $G$.)
	
	So we must verify each $f_i: U_i \to |\Delta^{n_i}| \times G^{n_i}$ is continuous, meaning we must verify the projections $U_i \to |\Delta^{n_i}|$ and $U_i \to G$ are each continuous. The map to $|\Delta^{n_i}|$ is continuous because the $D$-topology agrees with the usual topology on $|\Delta^n|$. To see that the projection to $G$ is continuous, for each $x \in U_i$, choose a smaller open set $V_i' \subset \subset U_i$ for which the restricted map $V_i' \to G$ extends to a smooth map from all of $|\Delta^k_e|$. This exhibits $U_i \to G$ as continuous near each $x \in U_i$, hence continuous.
 	
The map $|\Delta^{n_i}| \times G^{n_i} \to |\Delta^\omega| \times \prod_\omega G$ is obviously continuous, as is the projection map to $EG$ (by definition of the topology on $EG$---see Remark~\ref{remark. topology on EG and BG}), so we conclude that $f$ is continuous. This shows the middle vertical arrow in the statement of the lemma indeed lands in $\sing(EG)$.
	
	A similar argument shows that the right vertical arrow also lands in $\sing(BG)$.
	\end{proof}

\begin{proof}[Proof of Theorem~\ref{theorem. smooth and continuous homotopy groups BG}.]

Both rows in the commutative diagram of Lemma~\ref{lemma. fiber sequence maps} are Kan fibration sequences. (For the bottom row, this is classical; for the top row, see Remark~\ref{remark. bundles are kan fibrations}.) So consider the induced map of long exact sequences
	\eqnn
	\xymatrix{
	\ldots \ar[r] \ar[d]
	& \pi_n\sing^{C^\infty}(\widehat{G}) \ar[r]  \ar[d]
	& \pi_n\sing^{C^\infty}(\widehat{EG}) \ar[r]  \ar[d]
	& \pi_n\sing^{C^\infty}(\widehat{BG}) \ar[r] \ar[d]
	& \ldots \ar[d] \\
	\ldots \ar[r]
	& \pi_n\sing(G) \ar[r]
	& \pi_n\sing(EG) \ar[r]
	& \pi_n\sing(BG) \ar[r]
	& \ldots .
	}
	\eqnd
By Assumption~\ref{assumption on G}\ref{item. smooth pi n are pi n} and the contractibility of $EG$ (topologically, this is classical; diffeologically, this is Theorem~\ref{theorem. BG facts}\eqref{item. EG contractible}), we can apply the Five Lemma to conclude that the maps $\pi_n\sing^{C^\infty}(\widehat{BG}) \to \pi_n\sing(BG)$ are isomorphisms. Hence the map $\sing^{C^\infty}(\widehat{BG}) \to \sing(BG)$ is a weak homotopy equivalence.

Now let us check the claim regarding (the non-simplicially defined) smooth homotopy groups. By \ref{item. G is group}, $\widehat{G}$ is a diffeological group, hence $\widehat{EG} \to \widehat{BG}$ is a diffeological $G$-bundle (Theorem~\ref{theorem. BG facts}\eqref{item. EG is a bundle over BG}).
Thus we obtain a long exact sequence in $\pi^{C_\infty}_n$ (See 8.21 of \cite{iglesias-diffeology}). 

Further, similar techniques to the proof of Lemma~\ref{lemma. fiber sequence maps} yield induced homomorphisms $\pi_n^{C^\infty}(\widehat{BG}) \to \pi_n(BG)$,  $\pi_n^{C^\infty}(\widehat{EG}) \to \pi_n(EG)$. Moreover, these are compatible with the long exact sequences of homotopy groups. Thus we have a commutative diagram
	\eqnn
	\xymatrix{
	\ldots \ar[r] \ar[d]
	& \pi_n^{C^\infty}(\widehat{G}) \ar[r]  \ar[d]
	& \pi_n^{C^\infty}(\widehat{EG}) \ar[r]  \ar[d]
	& \pi_n^{C^\infty}(\widehat{BG}) \ar[r] \ar[d]
	& \ldots \ar[d] \\
	\ldots \ar[r]
	& \pi_n(G) \ar[r]
	& \pi_n(EG) \ar[r]
	& \pi_n(BG) \ar[r]
	& \ldots
	}.
	\eqnd
We know that $\widehat{EG}$ is smoothly contractible so the arrows $\pi_n^{C^\infty}(\widehat{EG}) \to \pi_n(EG)$ are isomorphisms. By \ref{item. smooth pi n are pi n} and Proposition~\ref{prop. smooth pi n is combinatorial}, so are the arrows $\pi_n^{C^\infty}(\widehat{G}) \to \pi_n(G)$. By using the Five Lemma again, we conclude that the smooth homotopy groups are naturally isomorphic to the continuous homotopy groups, proving the last statement of the theorem.
\end{proof}

\subsection{Example: \texorpdfstring{$\diff(Q)$}{Diff(Q)}}\label{section. diff Q}

\begin{defn}\label{defn. diff q diffeology}
Let $Q$ be a smooth manifold, not necessarily compact, let $\diff(Q)$ denote the set of smooth diffeomorphisms. We let $\widehat{\diff(Q)}$ denote the diffeological space as defined in Example~\ref{example. diff groups and their subgroups}.

We let $\diff(Q)_{\weak}$ and $\diff(Q)_{\strong}$ denote the topological spaces with topology inherited from the weak and strong topologies of $C^\infty(Q,Q)$, respectively.
\end{defn}

\begin{prop}\label{prop. diff Q is a group}
Let $Q$ be a smooth manifold, not necessarily compact. Then
\enum
\item $\diff(Q)_{\weak}$ and $\diff(Q)_{\strong}$ are both topological groups.
\item $\widehat{\diff(Q)}$ is a diffeological group.
\enumd
\end{prop}

\begin{proof}
The second claim is by design (see also Proposition~\ref{prop. diff X has one diffeology}). The first claim is classical. (Let us include a very brief sketch for the reader: The first claim holds because the chain rule and the inverse function theorem yield estimates on the $C^r$ norms of compositions and inverses, while continuity in both strong and weak topologies is detected by checking estimates along compact subsets of $M$.) 
\end{proof}

Though we will not use the following, we present it for the benefit of the reader: 

\begin{prop}\label{prop. topologies on Diff Q equivalent}
Let $Q$ be compact. Then the following topologies on $\diff(Q)$ are equivalent:
\enum
	\item $\diff(Q)_{\weak}$.
	\item $\diff(Q)_{\strong}$.
	\item The $D$-topology induced from the diffeological structure of $ \widehat{\diff(Q)}$ as a diffeological group (Definition~\ref{defn. diff q diffeology}).
	\item The subset topology inherited from the $D$-topology on $C^\infty(Q,Q)$.
	\item The $D$-topology induced by first endowing $\diff(Q) \subset C^\infty(Q,Q)$ with the subset diffeology.
\enumd
\end{prop}

\begin{proof}
The equivalence between 1 and 2 is classical when $Q$ is compact, as the weak and strong topologies coincide for $C^\infty(Q,Q)$. 
4 and 5 are equivalent by Corollary~4.15 of~\cite{christensen-sinnamon-wu}. Finally, 3 and 5 are equivalent because the diffeologies are equivalent by Proposition~\ref{prop. diff X has one diffeology}.
\end{proof}

%
%
%
%
%
%
	
%
%
%
%

Remark~\ref{remark. D top maps to weak} ensures that the restriction of a smooth map from an extended simplex is a continuous simplex in the weak topology. In fact, for compact manifolds, we have:

\begin{prop}\label{prop. diff q smooth approximation}
Let $Q$ be a compact smooth manifold. 
The map
 	\eqnn
	\sing^{C^\infty}(\widehat{\diff(Q)})
	\to
	\sing(\diff(Q)_{\weak})
	\eqnd
induced by \eqref{eqn. sing Coo to sing} is a homotopy equivalence of simplicial sets.
\end{prop}

\begin{proof}
Choose a basepoint $f_0 \in \diff(Q)$. In this paragraph, we review some basic definitions of simplicial homotopy groups: Recall that an element of $\pi_n (\sing(\diff(Q)_{\weak}), f_0)$ is determined by a map $\Delta^n \to \sing(\diff(Q)_{\weak})$ which restricts to the degenerate $(n-1)$-simplex determined by $f_0$ along every face of $\Delta^n$. Equivalently, this is the data of a continuous map $|\Delta^n| \to \diff(Q)_{\weak}$ sending the boundary $\del |\Delta^n|$ to $f_0$. A homotopy between $f$ and $g$ respecting $f_0$ is, by definition, a map from $\Delta^{n+1}$ to $\sing(\diff(Q)_{\weak}$ which restricts to $f$ and to $g$ along two faces, while restricting to the degenerate simplex $n$-simplex determined by $f_0$ along all other faces.

To show that the map on $\pi_n$ is a surjection, we must exhibit a homotopy from $f$ to a function $g$ for which $g$ is the restriction of a smooth map $|\Delta^n_e| \to \widehat{\diff(Q)}$.

Fix a point $(s,x) \in |\Delta^n| \times Q$; we let $f_s: Q \to Q$ denote the diffeomorphism determined by $s$. Now, the graph of $f_s$ (as a submanifold of $Q \times Q$) admits a tubular neighborhood that may be identified with (an open subset of the zero section of) the normal bundle of the graph; or, under the natural identification of the graph of $f_s$ with the domain of $f_s$, with the tangent bundle to $Q$.

In what follows, we let $\dim Q = d$.

With the definition of the weak topology in mind, we invoke the following consequence of the tubular neighborhood theorem:  For every $x \in Q$, for every chart $\phi: U \into \RR^{d}$ around $x$, for every chart $\psi: V \into \RR^{d}$ about $f_s(x)$, and for every compact ball $K$ with $x \in K \subset U$, there is some $\epsilon$ such that as long as a diffeomorphism $g_s: Q \to Q$ is $\epsilon$-close to $f_s$ along $K$, then the graph of $g_s$ may be written as the graph of some section of $TQ$ near $K$. More precisely, $\epsilon$-closeness guarantees that there exists some smooth section $G: Q \to TQ$ so that one has a commutative diagram
	\eqnn
	\xymatrix{
	K \ar[r]^{G|_K} \ar[d] & TQ \ar@{-->}[d] & Q \ar[l] \ar[dl]^{(\id_Q,f_s)} \\
	Q \ar[r]^{(\id_Q,g_s)} & Q \times Q.
	}
	\eqnd
Here, the vertical map $TQ \dashrightarrow Q \times Q$ (the dashed arrow means the map is defined on a small open neighborhood of the zero section) is the tubular neighborhood around the graph of $f_s$.

Because the map $f: |\Delta^n| \to \diff(Q)_{\weak}$ is assumed continuous, we know that there is some open subset $D \subset |\Delta^n|$ containing $s$ such that for all $s' \in D$, the graph of $f_{s'}$ may be represented as the graph of some smooth section $F_{s'}: Q \to TQ$ (where the graph of $f_s$ is the zero section). By the usual Whitney approximation theorem, we may choose a continuous homotopy from the continuous map $s' \mapsto F_{s'}$ to a smooth map $s' \mapsto G_{s'}$, while maintaining that the induced maps 
	\eqn\label{eqn. g s' diff}
	K \to Q,
	\qquad
	x \mapsto g_{s'}(x)
	\eqnd
are all $\epsilon$-close to the original $f_s|_{K}$. The closeness of the first derivatives guarantees that the $g_{s'}$ are all immersions (hence locally injective), while $C^0$-closeness guarantees that each $g_{s'}$ is globally injective; hence by compactness of $Q$, each $g_{s'}$ is a diffeomorphism. (Note that if $Q$ is not compact, we cannot guarantee to be open unless we are given that $|\Delta^n| \to \diff(Q)$ is continuous in the {\em strong} topology.)

So given $f: |\Delta^n| \to \diff(Q)_{\weak}$, choose once and for all a very fine subdivision of $|\Delta^n|$ so that for every subsimplex of the subdivision, we may find some $(K,U,F,\phi,\psi,\epsilon)$ as above such that $K$ is a compact ball of $|\Delta^n|$ containing the subsimplex. (By compactness of $|\Delta^n|$, we may find a {\em finite} such subdivision.) We proceed by induction on the dimensions of these subsimplices---at the $d$th step, we will perform a construction on all $d$-dimensional subsimplices. There is nothing to perform in dimension $d=0$. Now suppose we have a continuous map $j$ from a $d$-dimensional subsimplex of $|\Delta^n|$ to $\diff(Q)_{\weak}$. We may homotope $j$ so that $j$ is {\em collared} in the following sense. Writing $|\Delta^d|$ as the cone over $\del |\Delta^d|$ with cone point given by the barycenter, we have a continuous embedding $\del |\Delta^d| \times [0,1) \into |\Delta^d|$. We say $j$ is collared if for some $\delta>0$, $j$ is constant along $[0,\delta)$. (For example: If we further assume $j$ is collared when restricted to every subsimplex of $\del |\Delta^d|$, we in particular see that $j$ is constant in an open neighborhood of every vertex.) Note that we may always continuously homotope $j$ so that it is collared---for example, by deformation retracting $[0,\delta)$ to $\{0\}$---and without changing the value of $j$ along $\del|\Delta^d|$. Note that because our subdivision was fine enough, for every $s' \in |\Delta^d|$, $j_{s'}$ may be modeled as the graph of some section $G: |\Delta^d| \to TQ$. The homotoped $j$ retains this property.

So now homotope the continuous family of smooth functions $s' \mapsto J_{s'}$ to a smooth family of smooth functions $s' \mapsto G_{s'}$ as in \eqref{eqn. g s'}. By usual Whitney approximation, we can arrange so that $G_{s'} = J_{s'}$ in a neighborhood of $\del |\Delta^d|$ (where $J_{s'}$ is already smooth by the inductive hypothesis and the collaring assumption). In particular, if a subsimplex $|\Delta^d|$ intersects any boundary face of $|\Delta^n|$, we maintain that the associated family of diffeomorphisms is constant along that intersection.
 This determines smooth maps $g: |\Delta^d| \to \widehat{\diff(Q)}, s \mapsto g_s$, and completes the inductive step.

Gluing together the homotopies performed along steps $d=0$ to $d=n$, we obtain a homotopy $[0,1] \times |\Delta^n| \to \diff(Q)$ such that
\enum
\item The time-0 map $|\Delta^n| \cong \{0\} \times |\Delta^n| \to \diff(Q)$ is the original $f$ we began with,
\item The time-1 map $g: |\Delta^n| \cong \{1\} \times |\Delta^n| \to \diff(Q)$ has adjoint map $|\Delta^n| \times Q \to Q$ which is smooth along every subsimplex of the subdivision. Moreover, the collaring condition guarantees that $g$ is smooth along neighborhoods of any subsimplex, hence globally smooth. Finally, not only have we arranged that $g$ equals $f_0$ along $\del |\Delta^n|$, the collaring condition along $\del |\Delta^n|$ guarantees that $g$ is constantly equal to $f_0$ in some neighborhood of $\del |\Delta^n|$, hence we may extend $g$ (constantly, for example) to a smooth map $|\Delta^n_e| \to \widehat{\diff(Q)}$. 
\enumd
Now we note that, as usual, the continuous map $[0,1] \times |\Delta^n| \to \diff(Q)_{\weak}$ factors through the quotient map $[0,1] \times |\Delta^n| \to |\Delta^{n+1}|$, exhibiting the simplicial homotopy from $f$ to $g$. This completes the proof that the map on $\pi_n$ is a surjection.

A similar construction shows that, for any null-homotopy of $g$ in $\sing(\diff(Q)_{\weak})$, we may exhibit a null-homotopy in $\sing^{C^\infty}(\widehat{\diff(Q)})$. This shows that the map on $\pi_n$ is an injection as well.

Finally, since both simplicial sets are Kan complexes, we are finished.
\end{proof}

\begin{prop}
Let $Q$ be a compact smooth manifold.
Let $G=\diff(Q)$ be given the diffeology from Example~\ref{example. diff groups and their subgroups}, and the topology from Proposition~\ref{prop. topologies on Diff Q equivalent}. Then $\diff(Q)$ satisfies Assumption~\ref{assumption on G}.
\end{prop}

\begin{proof}
\ref{item. G is group} follows from Proposition~\ref{prop. diff Q is a group}.
\ref{item. smooth simplices are continuous} follows from the definition of the D-topology.
\ref{item. smooth pi n are pi n} is Proposition~\ref{prop. diff q smooth approximation}. 
\end{proof}

\section{Example: Liouville automorphisms}\label{section. liouville}

\subsection{Recollections}
By a {\em Liouville sector}, we mean a smooth manifold\footnote{From the definitions, it follows that $M$ must be non-compact if $\dim M >0$.} $M$ possibly with boundary, equipped with a smooth 1-form $\theta$ satisfying the following:
\enum
\item $d\theta$ is a symplectic form,
\item The boundary of $M$ defines a barrier, and
\item\label{item. liouville flow conical} The Liouville flow (i.e., the flow associated to the vector field symplectically dual to $\theta$) exhibits $M$ as a completion of a Liouville domain with convex boundary.
\enumd
We refer to 2.4 of~\cite{oh-tanaka-liouville-bundles}, and to~\cite{gps} for more details. In practice, we are only interested in an equivalence class of $\theta$---we declare $\theta$ and $\theta'$ to be equivalent if their difference is the de Rham derivative of a compactly supported smooth function on $M$.

\begin{defn}\label{defn. liouaut}
By a Liouville {\em automorphism}, we mean a smooth map $\phi: M \to M$ such that $\phi^*(\theta)$ is in the same equivalence class as $\theta$; that is, $\phi^*\theta - \theta = df$ for some compactly supported $f: M \to \RR$.
\end{defn}

\begin{remark}
Because $\omega = d\theta$ is a symplectic form, the equation $\theta = \omega(X_\theta,-)$ uniquely defines a vector field $X_\theta$ on $M$. This is often called the {\em Liouville vector field}, and its flow the {\em Liouville flow}. The condition \ref{item. liouville flow conical} above means the following: 
We can find a compact, codimension-zero submanifold $M^0 \subset M$ with boundary, for which
	\eqnn
	C = \overline{\del M^0 \setminus (M^0 \cap \del M)} \subset M
	\eqnd
is a contact manifold (possibly with boundary), and for which the Liouville flow of $M$ realizes $M$ as a union
	\eqn\label{eqn. conical flow}
	M = M^0 \bigcup_{C} \left(C \times \RR_{\geq 0}\right).
	\eqnd
(Here, $\RR_{\geq 0}$ parametrizes the Liouville flow.)
\end{remark}

\begin{remark}\label{remark. Liouville flow equivariant}
Note that any Liouville automorphism $\phi$ is equivariant with respect to the flow of $X_\phi$ outside of some compact set. 
Thus, Liouville automorphisms represent an automorphism group of a non-compact manifold still having some control ``near infinity.'' 
\end{remark}

Indeed, an alternative construction of the contact manifold $C$ is to consider only the geometry at $\infty$ of $M$---this is Giroux's construction of the ideal contact boundary. More details can be found in~\cite{giroux}.

\begin{notation}
We denote Giroux's ideal contact boundary by
	\eqnn
	\del_\infty M.
	\eqnd
\end{notation}
We also note that there is natural map
	\eqn\label{eqn. restriction to infinity}
	\Liouauto(M) \to \aut(\del_\infty M)
	\eqnd
from the set of Liouville automorphisms of $M$ to the set of contact diffeomorphisms of $\del_\infty M$.

\subsection{The topology and diffeology on Liouville automorphisms}
\begin{defn}\label{defn. Liouaut diffeology}
Fix a Liouville sector $M$. We let $\Liouauto(M)$ be the collection of Liouville automorphisms $\phi: M \to M$.

We topologize $\Liouauto(M)$ to have the coarsest topology for which:
\enum
\item $\Liouauto(M)$ contains the weak $C^\infty$ topology inherited from $C^\infty(M,M)$, and
\item The map \eqref{eqn. restriction to infinity} is continuous.
\enumd

Further, we let $\widehat{\Liouauto(M)}$ be the diffeological space with diffeology generated by the following requirements:
\enum
\item It contains the subset diffeology inherited from $C^\infty(M,M)$, and
\item The map \eqref{eqn. restriction to infinity} is smooth.
\enumd
For this last condition, we are endowing the collection of contact automorphsims $\aut(\del_\infty M)$ with the smooth diffeology restricted from $C^\infty(\del_\infty M, \del_\infty M)$.
\end{defn}

\begin{remark}\label{remark. Liou has a unique diffeology}
By the exact same proof as in Proposition~\ref{prop. diff X has one diffeology}, one sees that this diffeology renders $\widehat{\Liouauto(M)}$ a diffeological group---if $\phi: U \times M \to M$ is smooth, then so is the map $(u,x) \mapsto \phi_u^{-1}(x)$. Likewise, the codomain of \eqref{eqn. restriction to infinity}, considered as a diffeological space $\widehat{\aut(\del_\infty M)}$, is a diffeological group.
\end{remark}

Now let us discuss some topologies on the space of Liouville automorphisms.

\begin{notation}
Let $C^\infty(M,M)$ be the collection of smooth endomorphisms. By the weak topology on $\Liouauto(M)$, we mean the subset topology endowed from the weak topology on $C^\infty(M,M)$; and likewise for the strong topology. Let us denote these two topological spaces
	\eqnn
	\Liouauto(M)_{\weak}
	\qquad
	\text{and}
	\qquad
	\Liouauto(M)_{\strong}.
	\eqnd
\end{notation}

\begin{remark}
As usual, the weak topology has far fewer open sets, and the identity map  $\Liouauto(M)_{\strong} \to  \Liouauto(M)_{\weak}$ is continuous.

Moreover, there are plenty of continuous maps $j: S \to \Liouauto(M)_{\weak}$ that are not continuous in the strong topology. As an example, write $M = M^0 \bigcup (C \times \RR_{\geq 0})$ as in \eqref{eqn. conical flow}, and choose a point $x \in C$. Now choose the collection of compact sets $K_n = \{x\} \times \{n\} \subset C \times \RR_{\geq 0}$ and the real numbers $\epsilon_n = 1/n$. Some appropriate choices of charts on $M$ then define a basic open neighborhood $N$ of $\Liouauto(M)_{\strong}$. Because any $f \in \Liouaut(M)$ is $\RR_{\geq 0}$-equivariant outside a compact subset of $M$, one can choose a strong-topology neighborhood $N$ of $f$ to consist precisely of those $g$ in a weak-topology basic neighborhood for which $f$ and $g$ agree to all orders along $\{x\} \times [t_0,\infty) \subset M$ for $t_0=t_0(g)$ large enough. It is now easy to construct examples of continuous maps $j: S \to \Liouauto(M)_{\weak}$ for which $j^{-1}(N)$ is then not an open subset of $S$; for example, when $S = [0,1]$.
\end{remark}

\begin{prop}\label{prop. Liou is top group}
$\Liouauto(M)_{\weak}$ and $\Liouauto(M)_{\strong}$ are topological groups.
\end{prop}

\begin{proof}
For any smooth manifold $M$, $\diff(M)$ is always a topological group in either topology (Proposition~\ref{prop. diff Q is a group}). Because $\Liouauto(M) \subset \diff(M)$ is a subgroup, the definition of subset topology finishes the proof.
\end{proof}

\begin{remark}
Let us explain why we prefer to invoke $\Liouauto(M)_{\weak}$ (rather than the strong topology) in Definition~\ref{defn. Liouaut diffeology}. This is somewhat motivated by the covenient map $\Liouauto(M)_{D} \to \Liouauto(M)_{\weak}$ from the $D$-topology (in contrast, the strong topology maps out, into the $D$-topology), but is primarily motivated by the natural map
	\eqn\label{eqn. diff to Liou of cotangent}
	\diff(Q) \to \Liouauto(T^*Q)
	\eqnd
which sends a diffeomorphism of $Q$ to the induced exact symplectomorphism of the cotangent bundle $T^*Q$. (Here, a symplectomorphism $\phi$ is called exact if $\phi^*\theta = \theta$.) For sake of argument, suppose $Q$ is compact, so that the natural topologies on $\diff(Q)$ are all equivalent (Proposition~\ref{prop. topologies on Diff Q equivalent}). Then \eqref{eqn. diff to Liou of cotangent} is {\em not} continuous if the codomain is given the strong topology.

(For example, choose a neighborhood of $\id_{T^*Q} \in \Liouauto(T^*Q)_{\strong}$ determined in part by a sequence of compact subsets $K_n$, each farther away from the zero section, with parameter $\epsilon_n = 1/n$. The only diffeomorphism whose induced symplectomorphism has derivatives $\epsilon_n$-close to $\id_{T^*Q}$ along each $K_n$ is the identity diffeomorphism; but the singleton set $\{\id_Q\}$ is not open in $\diff(Q)$. This argument in fact shows that, for $Q$ compact or not, $\diff(Q)$ must be given the discrete topology if the map $\diff(Q) \to \Liouauto(T^*Q)_{\strong}$ is to be continuous.)

Long story short, the map~\eqref{eqn. restriction of simplices} is well-behaved for the topology and diffeology of Definition~\ref{defn. Liouaut diffeology}.
\end{remark}

Now we are in line to prove that \eqref{eqn. restriction of simplices} is a homotopy equivalence.

\begin{theorem}\label{thm. Liou smooth approx}
The map $\sing^{C^\infty}(\widehat{\Liouauto(M)})  \to \sing(\Liouauto(M)) $ induced by \eqref{eqn. restriction of simplices} is a homotopy equivalence of simplicial sets.
\end{theorem}

\begin{proof}
By Proposition 7 of~\cite{giroux}, the map
	\eqnn
	\Liouauto(M) \to \aut(\del_\infty M)
	\eqnd
is a Serre fibration. 

A similar argument as in loc. cit. shows that the induced map
	\eqnn
	\sing(\Liouauto(M)) \to \sing(\aut(\del_\infty M))
	\eqnd
is a Kan fibration of simplicial sets. With a straightforward verification of compatibility with diffeological structures, one also sees that the map
	\eqnn
	\sing^{C^\infty}(\widehat{\Liouauto(M)}) \to 
	\sing^{C^\infty}(\widehat{\aut(\del_\infty M)})
	\eqnd
is a Kan fibration of simplicial sets. 

Let $F$ denote the fiber of the map $\Liouauto(M) \to \aut(\del_\infty M)$. $F$ consists of those Liouville automorphisms that induce the identity map on $\del_\infty M$. Because Liouville automorphisms are, outside of a compact set, equivariant with respect to the Liouville flow (Remark~\ref{remark. Liouville flow equivariant}), we conclude that $F$ consists of {\em compactly supported} Liouville automorphisms of $M$.

Giving $F$ the induced topology and diffeology, we see (by a straightforward argument similar to the one in Proposition~\ref{prop. diff q smooth approximation} and using that $|\Delta^n|$ is compact) that the map
	\eqnn
	\sing^{C^\infty}(\widehat{F}) \to 
	\sing(F)
	\eqnd
is a weak homotopy equivalence of simplicial sets. The same result holds for the map
	\eqnn
	\sing^{C^\infty}(\widehat{\aut(\del_\infty M)}) \to \sing(\aut(\del_\infty M)).
	\eqnd
Thus, by the long exact sequence of homotopy groups associated to a Kan fibration, the result follows.
\end{proof}

\begin{prop}
Let $G=\Liouauto(M)$ be endowed with the weak topology, and the diffeology from Definition~\ref{defn. Liouaut diffeology}. Then $\Liouauto(M)$ satisfies Assumption~\ref{assumption on G}.
\end{prop}

\begin{proof}
\ref{item. G is group} follows from Proposition~\ref{prop. Liou is top group} and Remark~\ref{remark. Liou has a unique diffeology}.
\ref{item. smooth simplices are continuous} follows from the definition of the D-topology and Remark~\ref{remark. D top maps to weak}.
\ref{item. smooth pi n are pi n} is Theorem~\ref{thm. Liou smooth approx}.
\end{proof}

\section{Localization}

\begin{notation}[$\simp(\widehat{B})$]\label{notation. simp C oo}
Let $\widehat{B}$ be any diffeological space. We let $\simp(\widehat{B})$ denote the category of smooth, extended simplices of $\widehat{B}$. That is, an object of $\simp(\widehat{B})$ is the data of a smooth map $j: \widehat{|\Delta_e^k|} \to \widehat{B}$. (See Definition~\ref{defn. extended simplices} for $|\Delta_e^k|$.) A morphism is a commutative diagram
	\eqnn
	\xymatrix{
	\widehat{|\Delta_e^k|} \ar[rr] \ar[dr]_{j} && \widehat{|\Delta^{k'}_e|} \ar[dl]^{j'} \\
	& \widehat{B}
	}
	\eqnd
where the map $\widehat{|\Delta_e^k|} \to \widehat{|\Delta^{k'}_e|}$ is (induced by) an injective, order-preserving simplicial map.

Equivalently, let $(\Delta_{\inj})_{/\sing^{C^\infty}(\widehat{B})}$ be the slice category over $\sing^{C^\infty}(\widehat{B})$. Then
\eqnn
\simp(\widehat{B}) \cong (\Delta_{\inj})_{/\sing^{C^\infty}(\widehat{B})} .
\eqnd
\end{notation}

\begin{notation}[$\subdivision(X)$]\label{notation. subdivision}
More generally, if $X$ is any simplicial set, we let $\subdivision(X)$ denote the barycentric subdivision of $X$. This can be realized as follows: There exists a category of simplices of $X$, where an object is a map $f: \Delta^n \to X$ of simplicial sets, and a morphism from $f$ to $f'$ is a commutative diagram of simplicial sets
	\eqnn
	\xymatrix{
	\Delta^n \ar[rr] \ar[dr] && \Delta^{n'} \ar[dl] \\
	& X &
	}
	\eqnd
where $\Delta^n \to \Delta^{n'}$ is induced by an injection $[n] \to [n']$.

Equivalently, let $(\Delta_{\inj})_{/X}$ be the slice category over $X$. Then $\subdivision(X)$ is the nerve of this category:
\eqn\label{eq:subdivision}
\subdivision(X) = N((\Delta_{\inj})_{/X})
\eqnd

\end{notation}

\begin{remark}\label{remark. subdiv is simp}
We have  a natural isomorphism of simplicial sets
	\eqnn
	N(\simp(\widehat{B})) \cong \subdivision(\sing^{C^\infty}(\widehat{B})).
	\eqnd
\end{remark}

\subsection{Realizations of subdivisions}	
	The following lemma illustrates one power of localization: It turns (the nerve of) a strict category into a homotopically rich object.

	\begin{lemma}\label{lemma. kan completion is smooth sing}
	Let $X$ be a Kan complex. Then the Kan completion of $\subdivision(X)$ is homotopy equivalent to $X$.
	\end{lemma}
	
	We give two proofs. The second proof has the advantage that one sees an explicit map leading to the homotopy equivalence.
	
	\begin{proof}[Proof of Lemma~\ref{lemma. kan completion is smooth sing} using coCartesian fibrations.]	
	By construction \eqref{eq:subdivision}, $\subdivision(X)$ is the total space of a Cartesian fibration over $N(\Delta_{\inj})$ with discrete fibers; in particular, the opposite category is a coCartesian fibration over $\Delta^{\op}_{\inj}$.   This coCartesian fibration classifies the functor
		\eqnn
		\Delta^{\op} \to \sets \subset \inftyGpd \subset \inftyCat,
		\qquad
		[n] \mapsto \{[n] \to X\},
		\eqnd
	otherwise known as $X$. Recall that the colimit of a diagram of $\infty$-categories is computed by localizing the total space of the corresponding coCartesian fibration along coCartesian edges. Thus we have an equivalence of $\infty$-categories
		\eqnn
		\colim_{\Delta^{\op}} X
		\to
		\subdivision(X)[C^{-1}]		
		\eqnd
	where $C$ is the collection of coCartesian edges.
	Because the inclusion $\inftyGpd \subset \inftyCat$ admits a right adjoint, the colimit of this functor into $\inftyCat$ may be computed via the colimit in $\inftyGpd$, but this is of course the usual geometric realization because all of the relevant $\infty$-groupoids in the simplicial diagram are discrete. So the domain of the equivalence is homotopy equivalent to the singular complex of the geometric realization $X$; this is equivalent to $X$ because $X$ is a Kan complex.
		
	On the other hand, the localization on the righthand side is precisely the Kan completion of $\subdivision(X)^{\op}$ (because every edge is coCartesian), and hence the Kan completion of the non-opposite category.
	\end{proof}
	
Here is another proof.

\begin{notation}[$\max$]\label{notation. max}
Fix a simplicial set $X$. ($X$ need not be a Kan complex.) We denote by
	\eqnn
	\max: \subdivision(X) \to X
	\eqnd
the map of simplicial sets given by evaluating $j$ at the maximal vertex of $\Delta^n$.
\end{notation}

\begin{prop}\label{prop. barycentric is equivalent}
If $X$ is a Kan complex, $\max$ exhibits $X$ as the Kan completion of $\subdivision(X)$. More generally, even if $X$ is not a Kan complex, $\max$ is a weak homotopy equivalence, so the Kan completion of $X$ is homotopy equivalent to the Kan completion of $\subdivision(X)$.
\end{prop}

\begin{proof}
It suffices to show that the induced map of geometric realizations $|\max|: |\subdivision(X)| \to |X|$ is a homotopy equivalence; in fact, it is a homeomorphism. This is a classical result, as $\subdivision(X)$ is nothing more than the barycentric subdivision of $X$. See for example~III.4 of~\cite{goerss-jardine}.
\end{proof}

	\begin{remark}\label{remark. localizing simp is sing usually}
	If $B$ is any topological space, one can analogously define the strict category $\simp(B)$ of continuous simplices in $B$. The proof of Lemma~\ref{lemma. kan completion is smooth sing} adapts straightforwardly to show that the Kan completion of $N(\simp(B))$ is homotopy equivalent to $\sing(B)$.
	\end{remark}

\begin{proof}[Proof of Lemma~\ref{lemma. kan completion is smooth sing} using $\max$ map.]
Immediate from Proposition~\ref{prop. barycentric is equivalent}.
\end{proof}

\begin{proof}[Proof of Theorem~\ref{main theorem}.]
For brevity, given a simplicial set $X$, we let $|X|$ denote its Kan completion. We have the string of maps
	\begin{align}
	N(\simp(\widehat{BG}))
	& \to \subdivision(\sing^{C^\infty}(\widehat{BG})) \nonumber\\
	& \to \sing^{C^\infty}(\widehat{BG})\nonumber\\
	& \to \sing(BG).\nonumber
	\end{align}
The first map is an isomorphism by Remark~\ref{remark. subdiv is simp}.
The next map exhibits $|\sing^{C^\infty}(\widehat{BG})|$ as a Kan completion of $\subdivision(\sing^{C^\infty}(\widehat{BG})) $ by Lemma~\ref{lemma. kan completion is smooth sing}.
The last map is a weak homotopy equivalence by Theorem~\ref{theorem. smooth and continuous homotopy groups BG}. This shows that $\sing(BG)$ is the Kan completion of $N(\simp(\widehat{BG}))$; on the other hand, the localization of an $\infty$-category along all its morphisms is precisely the Kan completion. 
\end{proof}

\bibliographystyle{amsalpha}
\bibliography{biblio}

\end{document}